\documentclass[smallcondensed,smallextended,reqno,12t]{amsart}

\usepackage{diagrams}
\usepackage{amsmath}
\usepackage{comment}
\usepackage{amssymb}
\usepackage{amsthm}
\usepackage{enumerate}
\usepackage[mathscr]{eucal}
\usepackage{graphicx}
\usepackage{epsfig}
\usepackage{color}
\usepackage{tikz}
\usetikzlibrary{graphs}
\usetikzlibrary{positioning}

\theoremstyle{plain}
\newtheorem{theorem}{Theorem}
\newtheorem{lemma}[theorem]{Lemma}

\newtheorem{corollary}[theorem]{Corollary}

\newtheorem{observation}[theorem]{Observation}

\theoremstyle{definition}

\newtheorem{example}[theorem]{Example}
\newtheorem{remark}[theorem]{Remark}

\newcommand{\implik}{\Longrightarrow}

\newcommand{\ekviv}{\Longleftrightarrow}

\newcommand{\mnimpl}{\mbox{$
    \mathop{-\raise.8pt\hbox{\mathsurround=0pt$\!\scriptstyle\circ$}}$}}

\newcommand{\uv}[1]{``{#1}"}
\def\zruseno#1{}

\begin{document}

\title[Transition operators assigned to physical systems]%
{Transition operators assigned to physical systems}
\author[Ivan~Chajda   \and Jan Paseka]{Ivan~Chajda*  \and Jan~Paseka**}

\newcommand{\acr}{\newline\indent}

\address{\llap{*\,}Department of Algebra and Geometry\acr
                               Faculty of Science\acr
                               Palack\'y University Olomouc, 17. listopadu 12\acr
                               Olomouc, 771 46\acr
                               CZECH REPUBLIC}
\email{ivan.chajda@upol.cz}

\address{\llap{**\,}Department of Mathematics
					and Statistics\acr
					Faculty of Science\acr
					Masaryk University\acr
					{Kotl\'a\v r{}sk\' a\ 2}\acr
					Brno,611~37\acr
                               CZECH REPUBLIC}
\email{paseka@math.muni.cz}

\thanks{I. Chajda acknowledges the support by a bilateral project I 1923-N25 
New Perspectives on Residuated Posets  financed by  
Austrian Science Fund (FWF) 
and the Czech Science Foundation (GA\v CR). 
J. Paseka gratefully acknowledges Financial Support 
of the Czech Science Foundation (GA\v CR) un\-der the grant 
Algebraic, many-valued and quantum structures for uncertainty modelling 
No.~GA\v CR  15-15286S}


\begin{abstract} 
By a physical system we recognize a set of propositions about a given system 
with their truth-values depending on the states of the system.  
Since every  physical system can go from one state in another one, 
there exists a binary relation on the set of states describing this transition.  
Our aim is to assign to every such system an operator on the set of propositions which is fully 
determined by the mentioned relation. 
We establish conditions under which the given relation can be recovered 
by means of this transition operator.
\end{abstract}

\maketitle

\section*{Introduction}

In 1900, D.~Hilbert formulated his famous 23 problems, see e.g. \cite{hilbert} for detailes. 
In the problem number 6, he asked: \uv{Can physics be axiomatized?} It means that he asked 
if physics can be formalized and/or axiomatized for to reach a logically perfect system forming 
a basis of precise physical reasoning. This challenge was followed by G.~Birkhoff and 
J.~von~Neuman in 1930´s producing the so-called logic of quantum mechanics. We are 
going to addopt a method and examples of D.~J.~Foulis from \cite{foulis} and 
\cite{foulisexam}, however, we are not restricted to the logic of quantum mechanics. 
We are focused on a general situation with a physical system endowed by states 
which it can reach. Our goal is to assign to every such a system  the so-called transition 
operators completely determining its transition relation. Conditions under which this 
assignment works perfectly will be formulated.

We start with a formalization of a given physical system.
Every physical system is described by certain quantities and states through 
them it goes. From the logical point of view, we can formulate propositions 
saying what a quantity in a given state is. Through the paper we assume that 
these propositions can acquire only two values, namely either TRUE of FALSE. 
It is in accordance with reasoning both in classical physics and 
in quantum mechanics.

Denote by $S$ the set of states of a given physical system ${\mathcal P}$. 
It is given by the nature of  ${\mathcal P}$ from what state $s\in S$ 
the system ${\mathcal P}$ can go to a state  $t\in S$. 
Hence, there  exists a binary relation $R$ on $S$ such that 
$(s, t)\in R$. This process is called a {\em transition} of ${\mathcal P}$.

Besides of the previous, the observer of  ${\mathcal P}$ can formulate 
propositions revealing our knowledge about the system. The 
truth-values of these propositions depend on states.  For example, 
the proposition $p$ can be  true if the system ${\mathcal P}$ is in the state 
$s_1$ but false if ${\mathcal P}$ is in the state $s_2$. Hence, for each 
state $s\in S$ we can evaluate the truth-value of $p$, it is denoted 
by $p(s)$. As mentioned above, $p(s)\in \{0, 1\}$ where 
$0$ indicates the truth-value FALSE and $1$ indicates TRUE. 
The set of all truth-values for all propositions will be called the {\em table}. 

Denote by $B$ the set of propositions about the physical system ${\mathcal P}$ 
formulated by the observer. 
We can introduce the order $\leq$ on $B$ as follows: 
$$
\text{for}\ p,q\in B, p\leq q\ \text{if and only if}\ 
p(s)\leq q(s)\ \text{for all}\ s\in S.
$$
One can immediately check that the contradiction, i.e., the proposition 
with constant truth-value $0$, is the least element and the tautology, i.e., 
the proposition with the constant truth-value $1$ is the greatest element of the 
ordered set $(B;\leq)$; this fact will be expressed by the notation  
$(B;\leq, 0, 1)$ for the bounded ordered set of propositions about ${\mathcal P}$. 

We summarize our description as follows:
\begin{enumerate}[{\rm -}]
\item every physical system ${\mathcal P}$ will be identified with the 
couple $(B,S)$, where $B$ is the set of propositions about ${\mathcal P}$ 
and $S$ is the set of states on ${\mathcal P}$; 
\item the set $S$ is equipped with a binary relation $R$ such that 
${\mathcal P}$ can go from $s_1$ to $s_2$ provided 
$(s_1, s_2)\in R$; 
\item the set $B$ is ordered by values of propositions as shown above.
\end{enumerate}

To shed light on the previous concepts, 
let us addopt an example from \cite{foulisexam}.

\begin{example}\label{firef}\upshape  
For system ${\mathcal P}=(B,S)$ we have 
$B=\{0, l, r, n, f, b, l', r', n', f', b', 1\}$ and 
$S=\{s_1, s_2, s_3, s_4, s_5\}$ such that the table is as follows. 

\medskip

\begin{center}
\begin{tabular}{|c|c|c|c|c|c|c|}
\hline
&$0$&$l$& $r$& $n$& $f$& $b$\\ \hline
$s_1$&0&1&0&0&1&0\\ \hline
$s_2$&0&1&0&0&0&1\\ \hline
$s_3$&0&0&1&0&0&1\\ \hline
$s_4$&0&0&1&0&1&0\\ \hline
$s_5$&0&0&0&1&0&0\\ \hline
\end{tabular}\ .
\end{center}

\medskip

Remember that $p'$ is a complement of $p$, i.e., 
it has $0$ in the same instance where $p$ has 
$1$ and vice versa.

Then the order is vizualized in the following Hasse diagram.

 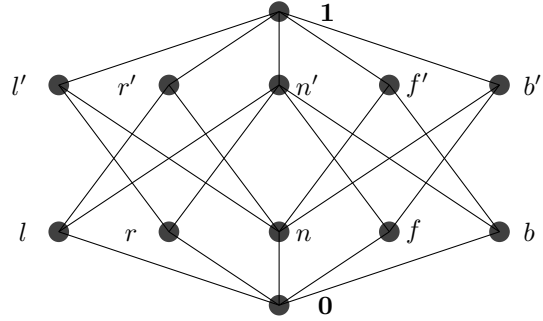
\begin{figure}[h]
\centering
\begin{tikzpicture}[scale=0.976583]
\coordinate [label=right:\phantom{ll}\hbox{$\phantom{ll}{\mathbf 1}$}] (1) at (0,4);
\coordinate [label=right:\phantom{ll}\hbox{$b$}\phantom{lll}] (b) at (3,1);
\coordinate [label=left:\phantom{lll}\hbox{$l$}\phantom{lll}] (l) at (-3,1);
\coordinate [label=left:\phantom{lll}\hbox{$l'$}\phantom{lll}] (l') at (-3,3);
\coordinate [label=left:\phantom{lll}\hbox{$r$}\phantom{lll}] (r) at (-1.5,1);
\coordinate [label=left:\phantom{lll}\hbox{$r'$}\phantom{lll}] (r') at (-1.5,3);
\coordinate [label=right:\phantom{ll}\hbox{$b'$}\phantom{lll}] (b') at (3,3);
\coordinate [label=right:\phantom{l}\hbox{$f'$}\phantom{lll}] (f') at (1.5,3);
\coordinate [label=right:\phantom{l}\hbox{$n$}\phantom{lll}] (n) at (0,1);
\coordinate [label=right:\phantom{l}\hbox{$f$}\phantom{lll}] (f) at (1.5,1);
\coordinate [label=right:\phantom{l}\hbox{$n'$}\phantom{lll}] (n') at (0,3);
\coordinate [label=right:\phantom{llll}\hbox{${\mathbf 0}$}] (0) at (0,0);
\draw (0) -- (l); 
\draw (0) -- (r); 
\draw (0) -- (b);
\draw (0) -- (f);
\draw (n') -- (b);
\draw (n') -- (r);
\draw (n') -- (l);
\draw (n') -- (f);
\draw (n') -- (1);
\draw (0) -- (n) ; 
\draw (f') -- (n);
\draw (f') -- (b);
\draw (f') -- (1);
\draw (b') -- (n);
\draw (b') -- (f);
\draw (b') -- (1);
\draw (r') -- (n);
\draw (r') -- (l);
\draw (r') -- (1);
\draw (l') -- (n);
\draw (l') -- (r);
\draw (l') -- (1);
\fill[black,opacity=.75] (0) circle (4pt);
\fill[black,opacity=.75] (1) circle (4pt);
\fill[black,opacity=.75] (l) circle (4pt);
\fill[black,opacity=.75] (r) circle (4pt);
\fill[black,opacity=.75] (f) circle (4pt);
\fill[black,opacity=.75] (n) circle (4pt);
\fill[black,opacity=.75] (n') circle (4pt);
\fill[black,opacity=.75] (b) circle (4pt);
\fill[black,opacity=.75] (f') circle (4pt);
\fill[black,opacity=.75] (b') circle (4pt);
\fill[black,opacity=.75] (r') circle (4pt);
\fill[black,opacity=.75] (l') circle (4pt);
\end{tikzpicture}
\caption{Logic of experimental propositions $\mathbb L$}\label{Fig2v}
\end{figure}

Finally, the relation $R$ on $S$ is given as follows

$$
R=\{(s_1,s_2), (s_2, s_3), (s_3,s_2), (s_3,s_5),  (s_4,s_3), (s_4,s_5), (s_5,s_5)\}. 
$$

It can be vizualized by the following graph of transition.

\tikzset{
my loop/.style={to path={
.. controls +(80:1) and +(100:1) .. (\tikztotarget) \tikztonodes}},
my state/.style={circle,draw}}

 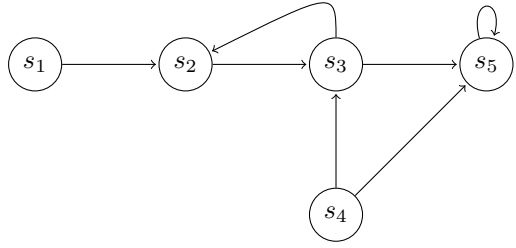
\begin{figure}[h]
\centering
\begin{tikzpicture}[shorten >=1pt,node distance=2cm,auto]
\node[my state] (s_1)  at (0,0) {$s_1$};
\node[my state] (s_2) [right of=s_1] {$s_2$};
\node[my state] (s_3) [right of=s_2] {$s_3$};
\node[my state] (s_4) [below of=s_3] {$s_4$};
\node[my state] (s_5) [right of=s_3] {$s_5$};
\draw [->] (s_1) -- (s_2);
\draw [->] (s_2) -- (s_3);
\draw [->] (s_3) -- (s_5);
\draw [->] (s_4) -- (s_3);
\draw [->] (s_4) -- (s_5);
\draw[->] (s_3) .. controls +(up:1cm)  ..   (s_2);
\path (s_5) edge [loop above]   (s_5);
\end{tikzpicture}
\caption{The transition graph of $R$}\label{Fig2tr}
\end{figure}  

\end{example}

Our task is as follows. We introduce an operator $T$ from $B$ into $2^{S}$ which is 
constructed by means of the relation $R$. The question is if this operator, 
called {\em transition operator}, bears all the information about system 
${\mathcal P}$ equipped with  the relation $R$. In other words, if the relation $R$  can be recovered by applying 
the operator $T$. In what follows, we will get conditions under which the 
transition operator has this property. Since these conditions are formulated 
in a pure algebraic way, we need to develop an algebraic background (see e.g. also in \cite{Blyth}) 
which is the purpose of the next section.

It is worth noticing that the transition operators will be constructed formally 
in a similar way as tense operators introduced by J.~Burges \cite{burges} for the classical logic 
and developped by the authors for several non-classical logics, see \cite{dyn}, \cite{dem} and \cite{doa}.

\section{Algebraic tools}

Let $S$ be a non-void set. Every subset $R\subseteq S\times S$ is called 
a {\em relation on $S$} and we say 
that the triple $(S, R)$ is a {\em transition frame}.
The fact that $(x, y)\in R$ for $x, y\in S$  is expressed by the notation $x R y$.

Let $(A;\leq)$ and $(B;\leq)$ be ordered sets, $f, g\colon A\to B$ mappings. 
We write $f\leq g$ if $f(a)\leq g(a)$, for all $a\in A$.
A mapping $f$ is called {\em order-preserving} or 
 {\em monotone} if $a, b \in A$ and $a \leq b$ together 
imply $f(a) \leq f(b)$, {\em antitone} if $a, b \in A$ and 
$a \leq b$ together imply $f(b) \leq f(a)$ 
and {\em order-reflecting}  if 
$a, b \in A$ and $f(a) \leq f(b)$ together imply $a \leq b$. 
A bijective order-preserving and order-reflecting mapping 
$f\colon A\to B$  is called an 
{\em isomorphism} and we say that the ordered sets 
$(A;\leq)$ and $(B;\leq)$ are {\em isomorphic}.

Let $(A;\leq)$ and $(B;\leq)$ be ordered sets. A mapping $f\colon A\to B$ is 
called {\em residuated} if there exists a mapping $g\colon B\to A$ 
such that 

$$f(a)\leq b\quad\text{if and only if}\quad a\leq g(b)$$
for all $a\in A$ and $b\in B$. 

In this situation, we say that $f$  and $g$ form a {\em residuated pair} or that 
the pair $(f,g)$ is a (monotone) {\em Galois connection}. 
The mapping $f$ is called a {\em lower adjoint of $g$} or a
{\em left adjoint of $g$}, the mapping $g$ is called an 
{\em upper adjoint of $f$} or a 
{\em right adjoint of $f$}. 

Galois connections can be described as follows, see e.g. \cite{Blyth}.

\begin{lemma}\label{GalCon}
Let $(A;\leq)$ and $(B;\leq)$ be ordered sets. Let  $f\colon A\to B$  and $g\colon B\to A$ be mappings. The following conditions are equivalent:
\begin{enumerate}
\item $(f,g)$ is a {Galois connection}.
\item $f$ and $g$ are monotone, $\mathop{id}_A\leq g\circ f$ and $ f\circ g\leq \mathop{id}_B$. 
\item $g(b)=\bigvee\{x\in A \mid f(x)\leq b\}$ and $f(a)=\bigwedge\{y\in B\mid a\leq g(y)\}$ 
for all $a\in A$ and $b\in B$.
\end{enumerate}
\end{lemma}

Moreover, it is possible to compose Galois connections: given Galois connections 
$( f,  g)$ between posets $\mathbf A$ and $\mathbf B$ and 
$(u, v)$ between posets $\mathbf B$ and $\mathbf C$, the pair 
$(u\circ  f,  g\circ  v)$ is also a Galois connection  between posets $\mathbf A$ and $\mathbf C$.

If an ordered set $\mathbf A$ has 
both a bottom and a top element, it will be called {\em bounded}; the appropriate 
notation for a bounded ordered set is $(A;\leq,0,1)$.  
Let $(A;\leq,0,1)$ and $(B;\leq,0,1)$ be bounded posets. A {\em morphism} 
$f\colon A\to B$ {\em of bounded  posets} 
is an order, 
top element and bottom element preserving map.

We can take the following useful result from \cite{dyn}.

\begin{observation}[\cite{dyn}]\label{obsik} Let $\mathbf A$ and $\mathbf M$  be 
bounded posets, $S$ a set and 
$h_{s}\colon A\to M, s\in S$, morphisms of bounded posets.
The following conditions are equivalent:
\begin{enumerate}
\item[{\rm(i)}] \(((\forall s \in S)\, h_{s}(a)\leq h_{s}(b))\implies a\leq
b\) for any elements \(a,b\in A\);
\item[{\rm(ii)}] The map $h\colon A \to M^{S}$ defined by 
$h(a)=(h_s(a))_{s\in T}$ for all $a\in A$ is order reflecting.
\end{enumerate}
\end{observation}
We then say that $\{h_{s}\colon A\to M; s\in S\}$ is a 
{\em full set of order-preserving maps with respect to} $M$. 
Note that we may in this case 
identify $\mathbf A$ with a bounded subposet   of  $\mathbf{M}^S$ since $h$ is an order reflecting 
morphism alias {\em embedding} of bounded posets. For any $s\in S$ and any 
$m=(m_t)_{t\in S}\in {M}^S$ we denote by $s(m)$ the $s$-th projection $m_s$. 
Note that $s(h(a))=h_s(a)$ for all $a\in A$ and all $s\in S$.

Consider  a complete lattice  $\mathbf M=(M;\leq,{}0, 1)$ and  
let  $(S,R)$ be a transition  frame.  
Further, let ${\mathbf{A}}=(A;\leq,{}0, 1)$  and 
${\mathbf{B}}=(B;\leq,{}0, 1)$ be  a 
bounded subposets of \/ $\mathbf M^{S}$.

Define 
mappings $P_R:A\to {M}^S$ and $T_R:B\to {M}^S$
  as follows: For all $b\in B$ and all $s\in S$,  
  
\begin{equation}\begin{array}{c}\mbox{$T_R(b)(s)=\bigwedge_{M}\{t(b)\mid s R t\} $}\phantom{.} \tag{$\ast$}
 \end{array}
\label{eqn:RTD}
\end{equation}
\noindent{}and, for all $a\in A$ and all $t\in S$,  

\begin{equation}
\begin{array}{c}
\mbox{${P}_R(a)(t)=\bigvee_{M}\{s(a)\mid s R  t\} $}{.} \tag{$\ast\ast$}
 \end{array}
\label{eqn:RPD}
\end{equation}

Then we say that ${T}_R$ ($P_R$) is an {\em upper transition operator} 
({\em lower transition operator})   {\em constructed by means of the  transition frame}  $(S,R)$, 
respectively.

For to answer our question given in introduction, we introduce certain binary relations 
induced by means of operators mentioned above.  Hence, let 
 $\mathbf{A}=({A};\leq, 0,1)$ and 
$\mathbf{B}=({B};$ $\leq, 0,1)$  be  bounded posets 
 with a full set $S$  of morphisms  of bounded 
 posets into a  non-trivial  complete lattice   
$\mathbf{M}$. 
We may assume that 
$\mathbf{A}$ 
and   $\mathbf{B}$ are bounded subposets of\/ $\mathbf{M}^{S}$.

Let $P:A\to B$ and $T:B\to A$ be morphisms of  posets.
Let us define the relations 
\begin{equation}R_T=\{(s, t)\in S\times S\mid (\forall b\in B) (s(T(b))\leq t(b))\} \tag{$\dagger$}
\label{eqn:RT}
\end{equation}
and 
\begin{equation}
R^{P}=\{(s, t)\in S\times S\mid (\forall a\in A) (s(a)\leq t(P(a)))\}.\tag{$\dagger\dagger$}
\label{eqn:RP}
\end{equation}

\bigskip

The relations $R_T$ and  $R^{P}$ 
on $S$  
will be called the 
{\em upper $T$-induced relation by ${\mathbf M}$} 
(shortly {\em $T$-induced relation by ${\mathbf M}$}) and 
{\em lower $P$-induced relation  by ${\mathbf M}$} 
(shortly {\em $P$-induced relation by ${\mathbf M}$}),
respectively.

Let $P\colon A\to {M}^{S}$ and $T\colon B\to {M}^{S}$ be morphisms of  posets, 
$R_T$ the $T$-induced relation by ${\mathbf M}$ and  
$R^P$ the $P$-induced relation by ${\mathbf M}$. Let us denote by ($\star$) 
the following condition: 

\begin{tabular}{@{}p{0.89\textwidth}@{\,}c}
& \\
\begin{minipage}{0.87\textwidth}
\begin{enumerate}
\item  for all $b\in B$ and for all $s\in S$, 
$s(T(b))=\bigwedge_{M}\{t(b)\mid  s R_G t\}$,
\item  for all $a\in A$ and for all $t\in S$, 
$t(P(a))=\bigvee_{M}\{s(a)\mid  s R^{P} t\}$.
\end{enumerate}
\end{minipage}&($\star$)\label{cstar}
\end{tabular}

\medskip

Now, let let $(S, R)$ be a transition frame and $T_R$, $P_R$  operators 
constructed by means of the  transition frame  $(S,R)$. We can ask under what conditions 
the relation $R$ coincides with the relation $R_{T_R}$ constructed as in 
($\dagger$)  or with the relation $R^{P_R}$ constructed as in 
($\dagger\dagger$). 
If this is the case we say that $R$ {\em is recoverable from} $T_R$ or  that 
$R$ {\em is recoverable from} $P_R$. We say that 
$R$ is {\em recoverable} if it is recoverable both from $T_R$  
and  $P_R$. The answer will be given in the next section. 

\section{The transition operator characterizing the physical system}

The connection between the relations $R$ and  $R_{T_R}$  or 
 $R$ and  $R^{P_R}$ is presented in the following lemma.

\begin{lemma}\label{reldgalreprest}
Let  $\mathbf{M}$ be a  non-trivial  complete lattice and $S$ a non-empty set. 
Let $\mathbf{A}$  and 
$\mathbf{B}$ be bounded subposets of\/ $\mathbf{M}^{S}$.  
Then 
\begin{enumerate}[\rm(a)] 
\item $R_1\subseteq R_2\subseteq S\times S $ implies $T_{R_2}\leq T_{R_1}$ 
and $P_{R_1}\leq P_{R_2}$. 
\item If $T_1, T_2:B\to M^{S}$ are order-preserving mappings such that 
$T_2\leq T_1$ and $T_2(1)=T_1(1)=1$ then 
$R_{T_1}\subseteq  R_{T_2}$.
\item If $P_1, P_2:A\to M^{S}$ are order-preserving mappings such that 
$P_1\leq P_2$ and $P_1(0)=P_2(0)=0$ then 
$R^{P_1}\subseteq  R^{P_2}$.
\item If $R\subseteq S\times S$ then $R\subseteq R_{T_R}\cap R^{P_R}$.  
\item  If $T:B\to M^{S}$ is an order-preserving mapping such that $T(1)=1$ then 
$T\leq T_{R_G}$. 
\item  If $P:A\to M^{S}$ is an order-preserving mapping such that $P(0)=0$ then 
$P_{R^{P}}\leq P$. 
\item Both the extremal relations $S\times S$ and $\emptyset$ are recoverable from $T_{S\times S}$ and $T_{\emptyset}$ 
or $P_{S\times S}$ and $P_{\emptyset}$, respectively.
\end{enumerate}

\end{lemma}
\begin{proof} (a): Assume that $R_1\subseteq R_2\subseteq S\times S $ and $s\in S$. Then 
$\{p(t)\mid s R_1 t\} \subseteq \{p(t)\mid s R_2 t\} $. It follows that 
$$\begin{array}{c}\mbox{$\bigwedge_{M}\{p(t)\mid s R_2 t\} $}\leq %
\mbox{$\bigwedge_{M}\{p(t)\mid s R_1 t\}$}.
\end{array}
$$
This yields that  $T_{R_2}\leq T_{R_1}$ and similarly $P_{R_1}\leq P_{R_2}$.

\noindent{}(b): Let $(s,t)\in R_{T_1}$. Then, for all $b\in B$, 
$s(T_2(b)))\leq s(T_1(b)))\leq t(b)$, i.e.,  $(s,t)\in R_{T_2}$.

\noindent{}(c): Let $(s,t)\in R^{P_1}$. Then, for all $a\in A$,  
$s(a)\leq t(P_1(a))\leq t(P_2(a))$, i.e., $(s,t)\in R^{P_2}$.

\noindent{}(d): Let $(s,t)\in R$. Then, for all $b\in B$, 
$s(T_R(b))=\bigwedge_{M}\{p(t)\mid s R t\} \leq t(b)$.  It follows that 
$(s,t)\in R_{T_R}$. Similarly, $(s,t)\in R^{P_R}$. 

\noindent{}(e): Let $b\in B$ and $s\in S$. Either the set 
$\{t\in S \mid (s,t)\in R_T\}=\emptyset$ in which case 
$s(T_{R_T}(b))=1$ and hence $s(T(b))\leq s(T_{R_T}(b))$ or 
$\{t\in S \mid (s,t)\in R_T\}\not=\emptyset$ in which case, for all 
$t\in \{t\in S \mid (s,t)\in R_T\}$, we have that $s(T(b))\leq t(b)$, 
i.e.,  $s(T(b))\leq s(T_{R_T}(b))$.

\noindent{}(f):  It follows by the same considerations as (e).

\noindent{}(g):   Clearly, from (d) we have 
$S\times S\subseteq R_{T_{S\times S}}\subseteq S\times S$, 
i.e., $R_{T_{S\times S}}=S\times S$. 
Since $T_{\emptyset}(b)=1$ for all $b\in B$, i.e., also 
$s(T_{\emptyset}(0))=1 >0= t(0)$ for all $s, t\in S$  it follows 
that $R_{T_{\emptyset}}=\emptyset$. 
Similarly, $R^{P_{S\times S}}=S\times S$ and 
 $R^{P_{\emptyset}}=\emptyset$. 
\end{proof}

In what follows, we are going to show that the transition
relations on $S$ and the transition operators on $\mathbf B$ form 
a Galois connection. This is important because in every 
Galois connection one of its components fully determines the second one and vice versa.

Let  $\mathbf{M}$ be a  non-trivial  complete lattice and $S$ a non-empty set. 
Let $\mathbf{B}$ be a bounded subposet of $\mathbf{M}^{S}$, 
$(\wp(S\times S);\subseteq, \emptyset, S\times S)$ be 
the  poset of all relations on $S$ 
 and $(\text{Map}_1(\mathbf{B}, \mathbf{M}^{S}); \sqsubseteq)$ 
be the poset of all order-preserving mappings 
$T\colon B\to M^{S}$ such that  $T(1)=1$ and $T_1\sqsubseteq T_2$ if and only if 
$T_2(b)\leq T_1(b)$ for all $b\in B$.  The smallest element of 
$(\text{Map}_1(\mathbf{B}, \mathbf{M}^{S}); \sqsubseteq)$ 
is the constant mapping  ${\mathbf 1}$ such that  ${\mathbf 1}(b)=(1)_{s\in S}$ for all $b\in B$.
Let us put, for all $R\in \wp(S\times S)$ and all  $T\in \text{Map}_1(\mathbf{B}, \mathbf{M}^{S})$, 
$\varphi(R)=T_R$ and $\psi(T)=R_T$.

\begin{theorem}\label{reldgalconnreprest}
Let  $\mathbf{M}$ be a  non-trivial  complete lattice and $S$ a non-empty set 
such that  $\mathbf{B}$ is a bounded subposet of $\mathbf{M}^{S}$. Then the couple 
$(\varphi, \psi)$ is a Galois connection between $(\wp(S\times S);$ $\subseteq, \emptyset, S\times S)$  
and $(\text{\upshape Map}_1(\mathbf{B}, \mathbf{M}^{S}); \sqsubseteq)$.
\end{theorem}
\begin{proof} From Lemma \ref{reldgalreprest} (a), (b) we know that $\psi$ and $\varphi$ are 
order-preserving mappings. It is enough to check that, 
for all $R\in \wp(S\times S)$ and all  $T\in \text{Map}_1(\mathbf{B}, \mathbf{M}^{S})$, 
$$
\varphi(R)\sqsubseteq T \ \text{if and only if} \ 
R\subseteq \psi(T).
$$

Assume first that $\varphi(R)\sqsubseteq T$ holds and let $(s,t)\in R$. 
Then, for all $b\in B$, we have $s(T(b))\leq s(\varphi(R)(b))= s(T_R(b))%
=\bigwedge_{M}\{t(b)\mid s R t\}\leq t(b)$. This yields 
that $(s,t)\in \psi(G)=R_T$.

Conversely, assume that $R\subseteq \psi(T)$ and let $b\in B$, $s\in S$. 
 Either the set $\{t\in S \mid (s,t)\in R\}=\emptyset$ in which case 
$s(T_{R}(b))=1$ which yields $s(T(b))\leq 1=s(\varphi(R)(b))$ or 
$\{t\in S \mid (s,t)\in R\}\not=\emptyset$. In the last case we have that  
$\{t(b)\in T \mid (s,t)\in R\}\not=\emptyset$ and by the definition of 
$\varphi(R)=T_R$ we have that 
$s(T_R(b))=\bigwedge_{M}\{t(b)\mid s R t\}\leq t(b)$ for all 
$t\in S$ such that $(s,t)\in R$. Since  $R\subseteq \psi(T)$ 
we have, for all $t\in S$ such that $(s,t)\in R$, that, for all $c\in B$, 
$s(T(c))\leq t(c)$. It follows that $s(T(c))\leq s(T_R(b))=s(\varphi(R)(b))$. But we have just 
proved that $\varphi(R)\sqsubseteq T$.
\end{proof}

\begin{remark}
We point out  that our recoverable relations from the respective 
upper transition operators are exactly fixpoints of the composition 
$\psi\circ \varphi\colon \wp(S\times S) \to \wp(S\times S)$.
\end{remark}

Dually, let  $\mathbf{M}$ be a  non-trivial  complete lattice and $S$ 
a non-empty set. 
Let $\mathbf{A}$ be a bounded subposet of $\mathbf{M}^{S}$, 
$(\wp(S\times S);\subseteq, \emptyset, S\times S)$ be 
the  poset of all relations on $S$  and 
$(\text{Map}_0(\mathbf{A}, \mathbf{M}^{S}); \leq)$ be the poset of all order-preserving mappings 
$P\colon A\to M^{S}$ such that  $P(0)=0$ and $P_1\leq P_2$ if and only if 
$P_1(a)\leq P_2(a)$ for all $a\in A$.  The smallest element of 
$(\text{Map}_0(\mathbf{A}, \mathbf{M}^{S}); \leq)$ 
is the constant mapping ${\mathbf 0}$ such that ${\mathbf 0}(a)=(0)_{s\in S}$ for all $a\in A$.
Let us put, for all $R\in \wp(S\times S)$ and all  $P\in \text{Map}_0(\mathbf{A}, \mathbf{M}^{S})$, 
$\Phi(R)=P_R$ and $\Psi(P)=R^{P}$.

The following is an immediate corollary of Theorem \ref{reldgalconnreprest} applied to relations 
on $S\times S$ and to the poset 
$(\text{Map}_1(\mathbf{A}^{op}, (\mathbf{M}^{op})^{S}); \sqsubseteq)$. To see it it is enough 
to note that $(\text{Map}_0(\mathbf{A}, \mathbf{M}^{S}); \leq)=%
(\text{Map}_1(\mathbf{A}^{op}, (\mathbf{M}^{op})^{S}); \sqsubseteq)$, 
$\Phi(R)=P_R=\phi(R^{-1})$ and $\Psi(P)=R^{P}=\psi(P)$.

\begin{theorem}\label{dualreldgalconn}
Let  $\mathbf{M}$ be a  non-trivial  complete lattice and $S$ a non-empty set 
such that  $\mathbf{A}$ is a bounded subposet of $\mathbf{M}^{S}$. Then the couple 
$(\Phi, \Psi)$ is a Galois connection between $(\wp(S\times S);$ $\subseteq, \emptyset, S\times S)$  
and $(\text{\upshape Map}_0(\mathbf{A}, \mathbf{M}^{S}); \leq)$.
\end{theorem}

The connection between relations induced by means of transition operators 
$T$ and $P$ is shown in the following lemma and theorem. Of 
course, we can identify the ordered sets $\mathbf{A}$ and $\mathbf{B}$ for 
to obtain the answer to our original task, i.e., for to obtain the relationship 
between transition operators and transition relations on a physical system ${\mathcal P}=(B,S)$.

 \begin{lemma}\label{xreldreprest}
Let  $\mathbf{M}$ be a  non-trivial  complete lattice and  $S$ a non-empty set 
such that  $\mathbf{A}$ and $\mathbf{B}$ are bounded subposets of\/ $\mathbf{M}^{S}$. 
 Let $P:A\to {M}^{S}$ and $T:B\to {M}^{S}$ 
be morphisms of posets such that, for all $a\in A$ and all $b\in B$, 
$$
P(a)\leq b\ \ekviv\ a\leq T(b).
$$
\begin{enumerate}[{\rm(a)}]
\item If $P(A)\subseteq B$ then $R_T\subseteq R^{P}$.
\item If $T(B)\subseteq A$ then $R^{P}\subseteq R_T$.
\item If $P(A)\subseteq B$ and $T(B)\subseteq A$ then $R_T= R^{P}$.
\end{enumerate}
\end{lemma}
\begin{proof} (a): Let $s, t\in S$ and $(s,t)\in R_T$. 
Let $a\in A$. We put $b=P(a)$. It follows that  $a\leq T(P(a))$ and hence 
$s(a)\leq s(T(P(a))\leq t(P(a))$, i.e., $(s,t)\in R^{P}$ and we have 
$R_T\subseteq R^{P}$. 

(b): It follows from the same reasonings as in (a). 

(c): It follows immediately from (a) and (b).

\end{proof}

Among other things, the following theorem shows that if 
a given transition relation $R$ can be recovered by the upper transition operator 
then, under natural conditions, it can be recovered by the lower 
transition operator and vice versa.
 
 \begin{theorem}\label{reldreprest}
Let  $\mathbf{M}$ be a  non-trivial  complete lattice and 
$(S,R)$ a transition frame. Let $\mathbf{A}$ and $\mathbf{B}$ be  
bounded subposets of\/ $\mathbf{M}^{S}$.  Let $P_R:A\to {M}^{S}$ and $T_R:B\to {M}^{S}$ 
be operators  {constructed by means of the transition frame} $(S,R)$, i.e., the respective 
parts of the condition  {\rm($\star$)} hold for the relation $R$.
Then, for all $a\in A$ and all $b\in B$, 
$$
P_R(a)\leq b\ \ekviv\ a\leq T_R(b).
$$
Moreover,  the following holds.
\begin{enumerate}[\rm(a)] 
\item Let for all $t\in S$ exist an element $b^t\in B$ such that, for all $s\in S$, $(s,t)\notin R$, we have 
$\bigwedge_{M}\{u(b^{t})\mid  s R u\}\not\leq t(b^{t})\not =1$. Then 
$R=R_{T_R}$. 
\item Let for all $s\in S$ exist an element $a^s\in A$ such that, for all $t\in S$, $(s,t)\notin R$, we have 
$\bigvee_{M}\{u(a_{s})\mid  u R t\}\not\geq s(a^{s})\not =0$. Then 
 $R=R^{P_R}$. 
 \item If $R=R_{T_R}$ and $T_R(B)\subseteq A$ then $R=R_{T_R}=R^{P_R}$. 
 \item If $R=R^{P_R}$ and $P_R(A)\subseteq B$ then $R=R_{T_R}=R^{P_R}$.
\end{enumerate}
\end{theorem}
\begin{proof} Clearly,   for all $a\in A$ and all $b\in B$, 

$$
\begin{array}{r c l}
P_R(a)\leq b\ 
&\ekviv&\ (\forall t\in S)( t(P_R(a))\leq t(b)) \\[0.05cm]
&\ekviv&\ (\forall s, t\in S)(sRt \implik s(a)\leq t(b)) \\[0.05cm]
&\ekviv&\ (\forall s\in S)(s(a)\leq s(T_R(b))) \ekviv a\leq T_R(b).
\end{array}
$$

We know from Lemma \ref{reldgalreprest}  (d) that $R\subseteq R_{T_R}$ and $R\subseteq R^{P_R}$.

It is enough to verify (a) and (c). The remaining parts (b) and (d) follow by dual considerations.

(a): Now, assume that $(s, t)\not\in R$.

Then either $\{u \in S \mid s R u\}=\emptyset$ or 
 $\{u \in S \mid s R u\}\not=\emptyset$. Suppose 
 $\{u \in S \mid s R u\}=\emptyset$. Then 
 $s(T_R(b^{t}))=\bigwedge_{M}\{u(b^{t})\mid  s R u\}=%
\bigwedge_{M} \emptyset=1\not\leq t(b^{t})\not =1$. 
 For the other case, we have  that 
$t\not\in \{u \in S \mid s R u\}$, i.e., $u\not= t$ and therefore 
 $s(T_R(b^{t}))=\bigwedge_{M}\{u(b^{t})\mid  s R u\}\not\leq t(b^{t})$. 
 Hence in both cases we get that $s(T_R(b^{t}))\not\leq t(b^{t})$, i.e., 
 $(s, t)\not\in R_{T_R}$. It follows that $R_{T_R}\subseteq R$, i.e., $R=R_{T_R}$. 
 
(c): Assume that $R=R_{T_R}$ and $T_R(B)\subseteq A$. From Lemma \ref{xreldreprest} we have that 
$R\subseteq R^{P_R}\subseteq R_{T_R}=R$ which yields the statement.
\end{proof}

The following two corollaries of Theorem \ref{reldreprest} show that if the 
set $B$ of propositions on the system $(B,S)$ is large enough, i.e., if it 
contains the full set  $\{0,1\}^S$  then the transition relation $R$ can be recovered by each of the transition operators.

 \begin{corollary}\label{fcorreldreprest}
Let  $\mathbf{M}$ be a  non-trivial  complete lattice and 
$(S,R)$ a transition frame. Let 
$\mathbf{B}$ be a bounded subposets of\/ $\mathbf{M}^{S}$ 
such that $\{0,1\}^{S}\subseteq B$.  Let $P_R:B\to {M}^{S}$ and $T_R:B\to {M}^{S}$ be 
operators  {constructed by means of the transition frame} $(S,R)$, i.e., the respective parts of the condition  
{\rm($\star$)} hold for the relation $R$.
Then $R=R^{P_R}=R_{G_R}$. 
\end{corollary}
\begin{proof}Let us we define, for all $t\in S$, an element $b^{t}\in {2}^{T}\subseteq B$ 
by 

$${u}(b^{t})=\begin{cases}0&\text{\upshape if}\  u=t,\\
                                                1&\text{\upshape if}\  u\not=t.
                                                \end{cases}$$
Then $b^{t}\in  B$ satisfies the assumption of part (a) from Theorem \ref{reldreprest}, i.e., 
$R=R_{T_R}$. Similarly, $R= R^{P_R}$.
\end{proof}
 
\begin{corollary}\label{correldreprest}
Let  $\mathbf{M}$ be a  non-trivial  complete lattice and 
$(S,R)$ a transition frame.   
Let   $\widehat{T}, \widehat{P}:{M}^S \to {M}^S$  be operators constructed by means of $(S,R)$. 
Then $R=R^{\widehat{P}}=R_{\widehat{T}}$.
\end{corollary}

We can illustrate previous results in the following two examples.

\begin{example}\label{example2}\label{ex2} Consider the system ${\mathcal P}=(B,S)$ 
of Example \ref{firef}.

From the  table we see the values of $B$ as follows:

\medskip

\begin{center}
\begin{tabular}{l l}
$0=(0,0,0,0,0)$,& $l=(1, 1,0, 0, 0)$,\\ 
$r=(0, 0, 1, 1, 0)$, &$n=(0, 0, 0, 0, 1)$, \\
$f=(1, 0, 0, 1, 0)$,& $b=(0,1, 1, 0, 0)$,\\ 
$l'=(0, 0, 1, 1, 1)$,& $r'=(1, 1, 0, 0, 1)$, \\
$n'=(1, 1, 1, 1, 0)$,& $f'=(0, 1, 1, 0, 1)$, \\
$b'=(1, 0, 0,$ $1, 1)$,& $1=(1,1,1,1,1)$.
\end{tabular}
\end{center}

\medskip

Using our formulas $(\ast)$ and $(\ast\ast)$, we compute the upper transition operator 
$T_R\colon B\to 2^{S}$ and the lower transition operator 
$P_R\colon B\to 2^{S}$ as follows: 

\medskip
\begin{center}
\begin{tabular}{c}
\begin{tabular}{l l}
$T_{R}(0)=0$,&$T_{R}(1)=1$,\\
 $T_{R}(l)=(1,0,0,0,0)$,& $T_{R}(l')=(0,0,0,1,1)$,\\
$T_{R}(r)=(0,1,0,0,0)$,& $T_{R}(r')=(1,0,1,0,1)$,\\
$T_{R}(n)=n$,&$T_{R}(n')=l$,\\
$T_{R}(f)=0$,&$T_{R}(f')=1$,\\
$T_{R}(b)=l$,& $T_{R}(b')=n$,\\
\end{tabular}\\
\begin{tabular}{l l}
\phantom{$T_{R}(l')=(0,0,0,1,1)$,}&\phantom{$T_{R}(l')=(0,0,0,1,1)$,}\\
$P_{R}(0)=0$,&$P_{R}(1)=f'$,\\
 $P_{R}(l)=b$,& $P_{R}(l')=f'$,\\
$P_{R}(r)=f'$,& $P_{R}(r')=f'$,\\
$P_{R}(n)=n$,&$P_{R}(n')=f'$,\\
$P_{R}(f)=f'$,&$P_{R}(f')=f'$,\\
$P_{R}(b)=f'$,& $P_{R}(b')=f'$.\\
\end{tabular}
\end{tabular}
\end{center}
\medskip 

Hence, $T_{R}(l), T_{R}(r), T_{R}(l')$ and $T_{R}(r')$ do not belong to $B$, they only belong to $2^{S}$. 
On the contrary, $P_R(B)\subseteq B$.

Now, we use $T_R$ for computing the binary relation $R_{T_R}$ (by the formula $(\dagger)$) 
and  $P_R$ for computing the binary relation $R^{P_R}$ (by the formula $(\dagger\dagger)$). 
We obtain that $R_{T_R}=R$, i.e., our given relation $R$ is recoverable by the transition operator $T_R$ 
and hence this operator is a characteristic of the system ${\mathcal P}=(B,S)$.

On the contrary, $R^{P_R}$ is given as follows 

$$
\begin{array}{r c@{\,}l}
R^{P_R}&=&\{(s_1,s_2),(s_1,s_3), (s_2, s_2), (s_2, s_3), (s_3,s_2), (s_3,s_3), (s_3,s_5),\\
& &(s_4,s_2),  (s_4,s_3), (s_4,s_5), (s_5,s_5)\}. 
\end{array}
$$

It can be vizualized by the following transition graph.

\tikzset{
my loop/.style={to path={
.. controls +(80:1) and +(100:1) .. (\tikztotarget) \tikztonodes}},
my state/.style={circle,draw}}

 \begin{figure}[h]
\centering
\begin{tikzpicture}[shorten >=1pt,node distance=2cm,auto]
\node[my state] (s_1)  at (0,0) {$s_1$};
\node[my state] (s_2) [right of=s_1] {$s_2$};
\node[my state] (s_3) [right of=s_2] {$s_3$};
\node[my state] (s_4) [below of=s_3] {$s_4$};
\node[my state] (s_5) [right of=s_3] {$s_5$};
\draw [->] (s_1) -- (s_2);
\draw [->] (s_2) -- (s_3);
\draw [->] (s_3) -- (s_5);
\draw [->] (s_4) -- (s_3);
\draw [->] (s_4) -- (s_2);
\draw [->] (s_4) -- (s_5);
\draw[->] (s_3) .. controls  ++(-0.5,-0.5)  ..   (s_2);
\draw[->] (s_1) .. controls  ++(1.5,1.5)  ..   (s_3);
\path (s_2) edge [in=230,out=270,loop]   (s_2);
\path (s_5) edge [loop above]   (s_5);
\path (s_3) edge [loop above]   (s_3);
\end{tikzpicture}
\caption{The transition graph of $R^{P_R}$}\label{Fig2xtr}
\end{figure}
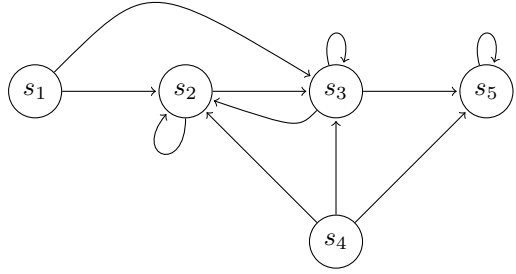  

One can easily see that $R\subseteq R^{P_{R}}$ but 
$R=R_{T_R}\not= R^{P_{R}}$ in accordance with our Theorem \ref{reldreprest} because 
the condition (b) from this theorem is not satisfied.

We can now use the same system ${\mathcal P}=(B,S)$  as above and, 
instead of the relation $R$ on $S$, we apply the new relation 
$R_2=R^{P_{R}}$. In other words, our system ${\mathcal P}=(B,S)$ will have the 
transition frame $(S, R^{P_{R}})$. After the computation of the new 
lower transition operator $P_{R_2}\colon B\to 2^{S}$ and the relation $R^{P_{R_2}}$ we see that 
now $R_2=R^{P_{R_2}}$.

Let us compute the new 
upper transition operator $T_{R_2}\colon B\to 2^{S}$ and the relation $R_{T_{R_2}}$. 
It follows that 
\medskip
\begin{center}
\begin{tabular}{c}
\begin{tabular}{l l}
$T_{R_2}(0)=0$,&$T_{R_2}(1)=1$,\\
 $T_{R_2}(l)=0$,& $T_{R_2}(l')=n$,\\
$T_{R_2}(r)=0$,& $T_{R_2}(r')=n$,\\
$T_{R_2}(n)=n$,&$T_{R_2}(n')=l$,\\
$T_{R_2}(f)=0$,&$T_{R_2}(f')=1$,\\
$T_{R_2}(b)=l$,& $T_{R_2}(b')=n$.\\
\end{tabular}
\end{tabular}
\end{center}
\medskip 
Since $T_{R_2}(B)\subseteq B$ and $P_{R_2}(B)\subseteq B$ we obtain from 
Theorem \ref{reldreprest} and Lemma \ref{xreldreprest} (c) 
that  $R_{T_{R_2}}=R^{P_{R_2}}=R_2$, i.e., 
$R_2$ is recoverable both by the upper transition operator $T_{R_2}$ and 
by the lower transition operator $P^{R_2}$.
\end{example} 

\begin{example} \label{ex3}Consider again  the system ${\mathcal P}=(B,S)$ 
of Example \ref{firef} but our relation $R$ will be now as follows:
$$
R=\{(s_1,s_2), (s_2, s_3), (s_3,s_2),  (s_3,s_4),  (s_4,s_3), (s_4,s_5), (s_5,s_3), (s_5,s_5)\}. 
$$

It can be visualized as in Figure \ref{Fig3tr}.

\tikzset{
my loop/.style={to path={
.. controls +(80:1) and +(100:1) .. (\tikztotarget) \tikztonodes}},
my state/.style={circle,draw}}

 \begin{figure}[h]
\centering
\begin{tikzpicture}[shorten >=1pt,node distance=2cm,auto]
\node[my state] (s_1)  at (0,0) {$s_1$};
\node[my state] (s_2) [right of=s_1] {$s_2$};
\node[my state] (s_3) [right of=s_2] {$s_3$};
\node[my state] (s_4) [below of=s_3] {$s_4$};
\node[my state] (s_5) [right of=s_3] {$s_5$};
\draw [->] (s_1) -- (s_2);
\draw [->] (s_2) -- (s_3);
\draw [->] (s_4) -- (s_3);
\draw [->] (s_4) -- (s_5);
\draw[->] (s_3) .. controls +(up:0.65cm)  ..   (s_2);
\draw[->] (s_5) --  (s_3);
\draw[->] (s_3) .. controls +(down:0.305cm) and (left:-0.103cm)  ..   (s_4);
\path (s_5) edge [loop right]   (s_5);
\end{tikzpicture}
\caption{The transition graph of $R$}\label{Fig3tr}
\end{figure}
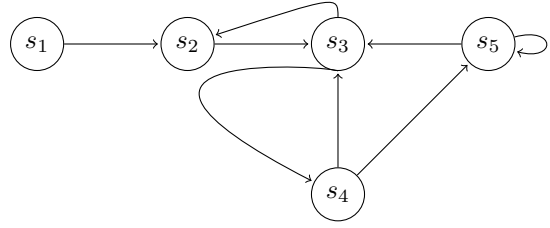  

For this relation the upper transition operator $T_R\colon B\to 2^{S}$ is as follows:

\medskip
\begin{center}
\begin{tabular}{l l}
$T_{R}(0)=0$,&$T_{R}(1)=1$,\\
 $T_{R}(l)=(1,0,0,0,0)$,& $T_{R}(l')=(0,0,0,1,1)$,\\
$T_{R}(r)=(0,1,0,0,0)$,& $T_{R}(r')=(1,0,0,0,0)$,\\
$T_{R}(n)=0$,&$T_{R}(n')=(1,1,1,0,0)$,\\
$T_{R}(f)=0$,&$T_{R}(f')=(1,1,0,1,1)$,\\
$T_{R}(b)=l$,& $T_{R}(b')=0$.\\
\end{tabular}
\end{center}
\medskip 

This upper transition operator $T_R$ differs from that of Example \ref{ex2} 
in values $T_{R}(n)$, $T_{R}(r')$, $T_{R}(n')$, $T_{R}(f')$ and 
$T_{R}(b')$.

Similarly, 
the lower transition operator $P_R\colon B\to 2^{S}$ is as follows:

\medskip
\begin{center}
\begin{tabular}{l l}
$P_{R}(0)=0$,&$P_{R}(1)=(0,1,1,1,1)$,\\
 $P_{R}(l)=b$,& $P_{R}(l')=(0,1,1,1,1)$,\\
$P_{R}(r)=(0,1,1,1,1)$,& $P_{R}(r')=f'$,\\
$P_{R}(n)=(0,0,1,0,1)$,&$P_{R}(n')=(0,1,1,1,1)$,\\
$P_{R}(f)=f'$,&$P_{R}(f')=(0,1,1,1,1)$,\\
$P_{R}(b)=(0,1,1,1,0)$,& $T_{R}(b')=f'$.\\
\end{tabular}
\end{center}
\medskip 

Now we compute  the  induced relation $R_{T_{R}}$, it is 
$$
R_{T_{R}}=\{(s_1,s_2), (s_2, s_3),  (s_3,s_1),   (s_3,s_2),  (s_3,s_3),   (s_3,s_4),  %
(s_4,s_3), (s_4,s_5), (s_5,s_3), (s_5,s_5)\}. 
$$

\tikzset{
my loop/.style={to path={
.. controls +(80:1) and +(100:1) .. (\tikztotarget) \tikztonodes}},
my state/.style={circle,draw}}

 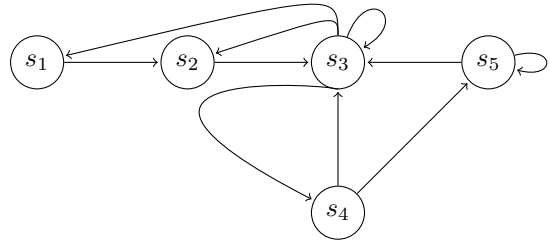
\begin{figure}[h]
\centering
\begin{tikzpicture}[shorten >=1pt,node distance=2cm,auto]
\node[my state] (s_1)  at (0,0) {$s_1$};
\node[my state] (s_2) [right of=s_1] {$s_2$};
\node[my state] (s_3) [right of=s_2] {$s_3$};
\node[my state] (s_4) [below of=s_3] {$s_4$};
\node[my state] (s_5) [right of=s_3] {$s_5$};
\draw [->] (s_1) -- (s_2);
\draw [->] (s_2) -- (s_3);
\draw [->] (s_4) -- (s_3);
\draw [->] (s_4) -- (s_5);
\draw[->] (s_3) .. controls +(up:0.65cm)  ..   (s_2);
\draw[->] (s_3) .. controls +(up:0.95cm)  ..   (s_1);
\draw[->] (s_5) -- (s_3);
\draw[->] (s_3) .. controls +(down:0.305cm) and (left:-0.103cm)  ..   (s_4);
\path (s_5) edge [loop right]   (s_5);
\path (s_3) edge [in=30,out=70,loop]   (s_3);
\end{tikzpicture}
\caption{The transition graph of $R_{T_R}$}\label{Fig4tr}
\end{figure}

One can easily see that $R\subseteq R_{T_{R}}$ but 
$R\not= R_{T_{R}}$ in accordance with our Theorem \ref{reldreprest} because 
the condition (a) from this theorem is not satisfied.

Similarly, we obtain the  induced relation $R^{P_{R}}$, it is 
$$
\begin{array}{r c@{\,}l}
R^{P_{R}}&=\{&(s_1,s_2), (s_1,s_3), (s_2, s_2), (s_2, s_3),  (s_3,s_2),  (s_3,s_3),   (s_3,s_4), \\ %
&&(s_4,s_2), (s_4,s_3), (s_4,s_5), (s_5,s_3), (s_5,s_5)\}.
\end{array} 
$$

\tikzset{
my loop/.style={to path={
.. controls +(80:1) and +(100:1) .. (\tikztotarget) \tikztonodes}},
my state/.style={circle,draw}}

 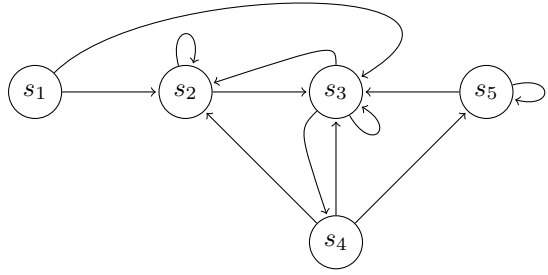
\begin{figure}[h]
\centering
\begin{tikzpicture}[shorten >=1pt,node distance=2cm,auto]
\node[my state] (s_1)  at (0,0) {$s_1$};
\node[my state] (s_2) [right of=s_1] {$s_2$};
\node[my state] (s_3) [right of=s_2] {$s_3$};
\node[my state] (s_4) [below of=s_3] {$s_4$};
\node[my state] (s_5) [right of=s_3] {$s_5$};
\draw [->] (s_1) -- (s_2);
\draw[->] (s_1) .. controls ++(1.5,1.5) and ++(2.5,1.5)  ..   (s_3);
\path (s_2) edge [loop above]   (s_2);
\draw [->] (s_2) -- (s_3);
\draw [->] (s_4) -- (s_2);
\path (s_3) edge [in=330,out=300,loop]   (s_3);
\draw [->] (s_4) -- (s_3);
\draw [->] (s_4) -- (s_5);
\draw[->] (s_3) .. controls +(up:0.65cm)  ..   (s_2);
\draw[->] (s_5) --  (s_3);
\draw[->] (s_3) .. controls ++(-0.5,-0.5)  ..   (s_4);
\path (s_5) edge [loop right]   (s_5);
\end{tikzpicture}
\caption{The transition graph of $R^{P_R}$}\label{Fig3trpr}
\end{figure}

We have $R\subseteq R^{P_{R}}$,  
$R\not= R^{P_{R}}$ and $R^{P_{R}}\not=R_{T_{R}}$, 
again in accordance with our Theorem \ref{reldreprest} because 
the condition (c) from this theorem is not satisfied.

However, we can now use the same system ${\mathcal P}=(B,S)$  as above and, 
instead of the relation $R$ on $S$, we apply the new relation 
$R_3=R_{T_{R}}$. In other words, our system ${\mathcal P}=(B,S)$ will have the 
transition frame $(S, R_{T_{R}})$. After the computation of the new 
upper transition operator $T_{R_3}\colon B\to 2^{S}$ and the relation $R_{T_{R_3}}$ we see that 
now $T_{R_3}=T_{R}$ and $R_3=R_{T_{R_3}}$. Let us compute 
 the new lower  transition operator $P_{R_3}\colon B\to 2^{S}$ and 
the new transition relation $R^{P_{R_3}}$. We have that  
\medskip
\begin{center}
\begin{tabular}{l l}
$P_{R_3}(0)=0$,&$P_{R_3}(1)=1$,\\
 $P_{R_3}(l)=b$,& $P_{R_3}(l')=1$,\\
$P_{R_3}(r)=1$,& $P_{R_3}(r')=f'$,\\
$P_{R_3}(n)=(0,0,1,0,1)$,&$P_{R_3}(n')=1$,\\
$P_{R_3}(f)=f'$,&$P_{R_3}(f')=1$,\\
$P_{R_3}(b)=n'$,& $T_{R_3}(b')=f'$.\\
\end{tabular}
\end{center}
\medskip 
Then $R^{P_{R_3}}=R_3\cup\{(s_1,s_3), (s_2,s_2), (s_4,s_2)\}$.

Similarly, applying the new relation $R_4=R^{P_{R}}$ we obtain 
the transition frame $(S, R^{P_{R}})$ for our system ${\mathcal P}=(B,S)$. 
Then we have that $R_4=R^{P_{R_4}}$. On the contrary, 
after the computation of  the new upper transition operator $T_{R_4}\colon B\to 2^{S}$ 
(see below) we have that $T_{R_4}\not =T_{R}$, $T_{R_4}\leq T_{R}$.
Moreover,  $R_4\not =R_{T_{R_4}}$ and $R_{T_{R_4}}=R_4\cup \{(s_3,s_1)\}$. 

\medskip
\begin{center}
\begin{tabular}{l l}
$T_{R_4}(0)=0$,&$T_{R_4}(1)=1$,\\
 $T_{R_4}(l)=0$,& $T_{R_4}(l')=0$,\\
$T_{R_4}(r)=0$,& $T_{R_4}(r')=0$,\\
$T_{R_4}(n)=0$,&$T_{R_4}(n')=(1,1,1,0,0)$,\\
$T_{R_4}(f)=0$,&$T_{R_4}(f')=(1,1,0,1,1)$,\\
$T_{R_4}(b)=l$,& $T_{R_4}(b')=0$.\\
\end{tabular}
\end{center}
\medskip 

\tikzset{
my loop/.style={to path={
.. controls +(80:1) and +(100:1) .. (\tikztotarget) \tikztonodes}},
my state/.style={circle,draw}}

 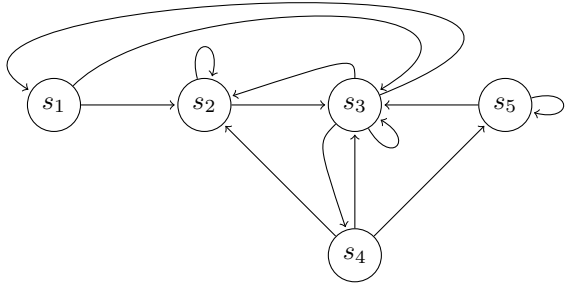
\begin{figure}[h]
\centering
\begin{tikzpicture}[shorten >=1pt,node distance=2cm,auto]
\node[my state] (s_1)  at (0,0) {$s_1$};
\node[my state] (s_2) [right of=s_1] {$s_2$};
\node[my state] (s_3) [right of=s_2] {$s_3$};
\node[my state] (s_4) [below of=s_3] {$s_4$};
\node[my state] (s_5) [right of=s_3] {$s_5$};
\draw [->] (s_1) -- (s_2);
\draw[->] (s_1) .. controls ++(1.5,1.5) and ++(2.5,1.5)  ..   (s_3);
\path (s_2) edge [loop above]   (s_2);
\draw [->] (s_2) -- (s_3);
\draw [->] (s_4) -- (s_2);
\path (s_3) edge [in=330,out=300,loop]   (s_3);
\draw [->] (s_4) -- (s_3);
\draw [->] (s_4) -- (s_5);
\draw[->] (s_3) .. controls +(up:0.65cm)  ..   (s_2);
\draw[->] (s_5) --  (s_3);
\draw[->] (s_3) .. controls ++(-0.5,-0.5)  ..   (s_4);
\draw[->] (s_3) .. controls ++(4.98,2.0)  and ++(-2.5,1.5) ..   (s_1);
\path (s_5) edge [loop right]   (s_5);
\end{tikzpicture}
\caption{The transition graph of $R^{P_{R_3}}=R_{T_{R_4}}=R^{P_{R_5}}$}\label{Fig3trpr}
\end{figure}

Put now $R_5=R_{T_{R_4}}$. Then $T_{R_5}=T_{R_4}$ and $R_5=R_{T_{R_5}}$ 
when computing the new upper transition operator $T_{R_5}\colon B\to 2^{S}$ and 
the new transition relation $R_{T_{R_5}}$. Let us compute the respective 
lower  transition operator $P_{R_5}\colon B\to 2^{S}$ and 
the new transition relation $R^{P_{R_5}}$. We obtain that 
\medskip
\begin{center}
\begin{tabular}{l l}
$P_{R_5}(0)=0$,&$P_{R_5}(1)=1$,\\
 $P_{R_5}(l)=b$,& $P_{R_5}(l')=1$,\\
$P_{R_5}(r)=1$,& $P_{R_5}(r')=f'$,\\
$P_{R_5}(n)=(0,0,1,0,1)$,&$P_{R_5}(n')=1$,\\
$P_{R_5}(f)=f'$,&$P_{R_5}(f')=1$,\\
$P_{R_5}(b)=n'$,& $T_{R_5}(b')=f'$.\\
\end{tabular}
\end{center}
\medskip 
It follows that $P_{R_3}=P_{R_5}$ and 
$R^{P_{R_5}}=R_5=R_{T_{R_4}}=R_{T_{R_5}}=R^{P_{R_3}}$, i.e., our construction stops now.
\end{example}

%


\begin{thebibliography}{99}

\bibitem{Blyth} {BLYTH, T.S.:}  \textit{Lattices and ordered algebraic structures}, 
Springer-Verlag London Limited, 2005.

  \bibitem{burges}
  BURGES, J.:  \textit{Basic tense logic}.
  In: Handbook of Philosophical Logic, vol. II
  (D. M. Gabbay, F. G\"{u}nther, eds.),
  D. Reidel Publ. Comp., 1984, pp.~89--139.



\bibitem{dyn} CHAJDA, I.---PASEKA, J.: \textit{Dynamic Effect Algebras and their Representations}, Soft Computing \textbf{16}, (2012), 1733--1741.


 

\bibitem{dem} CHAJDA, I.--- PASEKA, J.:  \textit{Tense Operators and Dynamic De Morgan Algebras},  
 In: Proc. 2013 IEEE 43rd Internat. Symp. Multiple-Valued Logic,  
Springer,  (2013), 219--224.

\bibitem{doa} CHAJDA, I.--- PASEKA, J.: {\it Dynamic Order Algebras as an Axiomatization  of Modal and Tense Logics}, 
{International Journal of Theoretical Physics}, 
doi:  10.1007/s10773-015-2510-9. 


 %



  \bibitem{foulis}
  FOULIS, D. J.---BENNETT, M. K.:
  \textit{Effect algebras and unsharp quantum logics},
  Found. Phys. \textbf{24}, (1994), 1325--1346. 

  \bibitem{foulisexam}  FOULIS, D. J.:
\textit{A half-century of  quantum  logic what have we learned?}, 
In 
 Proc. Conf. Einstein Meets Magritte, Brussels,
June, 1995.

\bibitem{hilbert} HILBERT,  D. : 
\textit{Mathematical Problems}, 
Bull. Amer. Math. Soc. \textbf{8} (1902), 437--479.





\end{thebibliography}
\end{document}